\documentclass[11pt,a4paper]{article}
\usepackage[utf8]{inputenc}

\usepackage{amsmath}
\usepackage{amsfonts}
\usepackage{amssymb}
\usepackage{mathrsfs}
\usepackage{amsthm}
\usepackage{parskip}

\usepackage{csquotes}

\usepackage{theoremref}

\usepackage[english]{babel}

\usepackage[left=2cm,right=2cm,top=3cm,bottom=2cm]{geometry}

\usepackage{biblatex}
\author{Gisel Mattar Marriaga}
\title{The Principal Symbol Map for\\ Lagrangian Distributions with Complex Phase}
\date{{\small Mathematical Institute\\  University of Göttingen}}

\newcommand{\R}{\mathbb{R}}
\newcommand{\C}{\mathbb{C}}
\newcommand{\smooth}{\mathscr{C}^\infty} 
\newcommand{\Z}{\mathbb{Z}}

\newcommand{\supp}[1]{\operatorname{supp}(#1)}

\newcommand{\LR}[1]{\left(#1\right)} 
\newcommand{\wf}[1]{\operatorname{WF}\left(#1\right)} 

\theoremstyle{plain}
\newtheorem{thm}{Theorem}[section]
\newtheorem{lemma}[thm]{Lemma}

\newtheorem{prop}[thm]{Proposition}

\theoremstyle{definition}
\newtheorem{defi}[thm]{Definition}

\newtheorem{ass}[thm]{Assumption}
\newtheorem{rem}[thm]{Remark}

\numberwithin{equation}{section}

\begin{document}
\maketitle

\begin{abstract}
Fourier integral operators with complex phase function are an important tool in the analysis of partial differential equations. A rigorous study of these operators can be found in  \cite{melin1975fourier}. There, the authors give a rough description of the principal symbol map  of a Lagrangian distribution after transverse composition. The present paper provides an alternative construction of the principal symbol map, which allows us to compute the principal symbol after clean composition. 
\end{abstract}

\section{Introduction}This paper studies Lagrangian distributions with complex phase functions. These are distributions $A\in\mathcal{D}(X,\Omega^{\frac{1}{2}})$, which are microlocally of the form 
\begin{align}\label{OscInt}
\int e^{i\phi(x,\theta)}a(x,\theta)\ d\theta,\quad (x,\theta)\in X\times\R^n,
\end{align}
and the phase function $\phi$ is assumed to have non-negative imaginary part, $\Im \phi\geq 0$. In particular, we are concerned with the principal symbol of such distributions, and the composition  of Fourier integral operators under the assumption of clean intersection.

Specifically, our main theorem provides an explicit description of the principal symbol $\sigma^m(A)$ of a Lagrangian distribution $A\in I^m_{\textup{cl}}(X,\Lambda,\Omega^{\frac{1}{2}})$.
\begin{thm}\thlabel{PS2}
 Let $A=I(\phi,a)$, with $\phi(x,\theta)$ a complex-valued non-degenerate phase function and $a(x,\theta)\in S^{m+(n-2N)/4}_{\textup{cl}}$. Then, the principal symbol $ \sigma^m(A)$ of $A\in I^m_{\textup{cl}}(X,\Lambda,\Omega^{\frac{1}{2}})$ defines a section of the virtual line bundle $\mathscr{L}$ over $\Lambda$. In admissible local coordinates, $\sigma^m(A)$ takes the form
\begin{align}
    \sigma^m(A)\sim\widetilde{a_0} \sqrt{d\phi}\in S^{(m+n/4)}(\Lambda,\mathscr{L}),
\end{align}
where $a_0$ is the top order term of the asymptotic expansion of $a$.
\end{thm}

Additionally, we consider Fourier integral operators with complex phase, that is operators with distributional kernel in the class $I^m_{\textup{cl}}(X,\Lambda,\Omega^{\frac{1}{2}})$ of Lagrangian distributions with complex phase. Furthermore, we study to  the composition of two such operators under the assumption of \emph{clean intersection} and compute the principal symbol of the resulting distribution. Then, the second main result of the paper is given by the following theorem.

\begin{thm}\thlabel{main}
Let $A_1\in I^{m_1}_{\textup{cl}}(X\times Y, C_1';\Omega^{\frac{1}{2}})$, $A_2\in I^{m_2}_{\textup{cl}}(Y\times Z, C_2';\Omega^{\frac{1}{2}})$ be such that the clean composition $B=A_1\circ A_2$ defines a distribution in  $I^{m_1+m_2+e/2}_{\textup{cl}}(X\times Z, (C_1\circ C_2)';\Omega^{\frac{1}{2}})$. Then, \[\sigma(B)^{m+e/2}\sim\int_{C_\gamma} \left((a_1)_0(a_2)_0(\theta^2+\sigma^2)^{\frac{-n_Y}{2}} \sqrt{d\Phi}\right)\ d\omega''\in S^{(m-e/2+n/4)}(\Lambda,\mathscr{L}),\] with $m=m_1+m_2$, $n=n_X+n_Z$, $e$ is the excess of the intersection and $C_\gamma$ is the compact image of a point $\gamma\in (C_1\circ C_2)_\R$ in $(C_{1\R}\times C_{2\R})\cap D$. 
Here  $(a_1)_0$, $(a_2)_0$ are the principal parts of the amplitudes of $A_1$ and $A_2$, respectively.
\end{thm}

The reader familiar with the real-valued theory, where the phase function is assumed to be non-degenerate in the sense of Hörmander \cite{Hormander1971FIOsI}, might recognize the above description of the principal symbol. Yet, the meaning is slightly different, because of the geometrical objects involve. When $\phi$ is real-valued, $\Lambda\subset(T^*X\setminus 0)$ is a Lagrangian manifold that can be parametrized by the phase function $\phi$, and $\mathscr{L}$ is the tensor product of the so-called Maslov line bundle and the space of half-densities $\Omega^{\frac{1}{2}}$.  However, this cannot be directly carried to the complex domain. For this paper, we followed an approach proposed by \citeauthor{melin1975fourier} in  \cite{melin1975fourier} and \cite{melin1976parametrices}. The authors use \emph{almost analytic extensions} to define $\Lambda$ as an abstract, complex object called \emph{positive Lagrangian manifold}, even tough it is not a Lagrangian manifold in the usual sense. In the same way, $\mathscr{L}$ is only a virtual line bundle over $\Lambda$. Thus, here $S^{(d)}(\Lambda,\mathscr{L})$, $d=d_1+d_2$, denotes the space of almost analytic functions $f\sim bs$, such that $s$ is a section $\mathscr{L}$  which homogeneous of degree $d_1$, and $b$ is an homogeneous function of degree $d_2$ in $\theta$. 

It should be noted that, working with almost analytc extensions introduces some errors, and the constructions are highly technical. For this reason, the authors in \cite{melin1975fourier} define the principal symbol of a Lagrangian distribution $A\in I^m_{\textup{cl}}(X,\Lambda,\Omega^{\frac{1}{2}})$ as the pre-image of $A$ under a, rather complicated, bijection $\mathcal{P}$, that maps  $\Gamma^{m+n/4}(\Lambda,\mathscr{L})$, the space of homogeneous sections of $\mathscr{L}$, into the set of oscillatory integrals \eqref{OscInt} with amplitude $a$ equal to some function defined on $\C^N\times\C^n$ and homogeneous of degree $m$.

To improve this description, we extend to the complex-valued case, the method by \citeauthor{duistermaat1973fourier} in \cite{duistermaat1973fourier} to compute the principal symbol in the real-valued theory. First, the leading order term of the asymptotic expansion gets identified with the principal symbol of a given Lagrangian distribution. Then, we show that the resulting expression defines a section of $\mathscr{L}$, which is equivalent to the section $\mathcal{P}^{-1}(A)$. This description allows us to see that, similar to the real-valued case, the principal symbol map $\sigma^m$ fits into a short exact sequence
\begin{align*}
0\to I^{m-1}_{\textup{cl}}(X,\Lambda;\Omega^{\frac{1}{2}}) \to I^m_{\textup{cl}}(X,\Lambda;\Omega^{\frac{1}{2}})\xrightarrow{\sigma^{m}}S^{(m+n/4)}(\Lambda,\mathscr{L}) \to 0.
\end{align*}

Since the publication of \cite{melin1975fourier} and \cite{melin1976parametrices}, other authors have contribute to the theory. For instance, in \cite{hormander2009analysis} \citeauthor{hormander2009analysis} proposes an alternative approach, using \emph{complex Lagrangian ideals} instead of almost analytic extensions. He also proves necessary and sufficient conditions for an operator of order zero to be $L^2$ continuous \cite{hormander19832}. In \cite{treves1980introduction}, \citeauthor{treves1980introduction} shows a general method for solving complex eikonal equations, and thus finding the phase function of the solution operators associated to hyperbolic problems. But, to the best of our knowledge, there is no explicit formulation of the principal symbol of a Lagrangian distribution with complex phase in the literature. We were also not able to find any mention of complex-valued clean phase functions or clean composition. With this paper, we attempted to fill these gaps. 

The present document is divided into three sections. The first of them consist of an overview of the theory of Lagrangian distributions with complex phase. We collect, without proofs, the results from \cite{melin1975fourier} that are necessary to achieve our goals. In the next section, we present our construction of the principal symbol map (\thref{PS2}), which is complementary to the one in \cite{melin1975fourier}. In the last section, we consider the case of \emph{clean composition} and proof our main result, \thref{main}.

\section{Previously known results}In this section, a summary of the theory of Lagrangian distributions with complex phase is presented. This a broad theory, so we focused on the definitions and theorems necessary to understand our main results. The content is completely taken from \cite{melin1975fourier}, but some of the statements have been reformulated to facilitate the reading process. For a complete presentation, the interested reader can see the original document \cite{melin1975fourier}. For those interested in the theory due its applications, \cite{treves1980introduction} is a good source. There, a nice exposition of the theory can be found, however, it does not include the construction of the principal symbol map. 

We begin by clarifying some notation and recalling the almost analytic machinery. Let $\overline{\partial}_z$ denote the Cauchy-Riemann operator in $\C$. If $f$ is a smooth function in $\C^n$, denote by $\overline{\partial}f$ the sum
\[\overline{\partial}f=\sum \overline{\partial}_{z_j}f\ d\bar{z}_j.\]

\begin{defi}
Let $\Omega\subseteq \C^n$ be an open set. We say that:
\begin{enumerate}
    \item $f\in\smooth(\Omega)$ is almost analytic if $\overline{\partial}f$ vanish to the infinite order near $\Omega_\R=\Omega\cap\R^n.$ This means that, for all $z\in\C^n$, there exist a constant $C_N>0$ such that
    \[|\overline{\partial}f(z)|\leq C_N |\Im z|^N,\ \forall N\in\mathbb{Z}_+.\]
    \item $f_1,f_2\in\smooth(\Omega)$ are equivalent if $f_1-f_2$ is almost analytic. In this case, we write $f_1\sim f_2$.
    \item An almost analytic extension of a function $f\in\smooth(\Omega_\R)$, is an almost analytic function $\widetilde{f}\in\smooth(\Omega)$ such that $\widetilde{f}\vert_{\Omega_\R}=f$.
\end{enumerate}
\end{defi}

As the equivalence relation suggests, when working with almost analytic extensions, we care about the equivalence class, not about any particular representative. It may be useful to think about this as the germs, in the sense of algebraic topology, of complex-valued functions functions at real points.

It is not hard to see that every $f\in\mathscr{S}(\R)$ admits an almost analytic extension, see for example \cite[Theorem 3.6]{zworski2012semiclassical}. In \cite{melin1975fourier}, a more general result is proven: Given a function $a\in S^m_{\textup{cl}}(\Gamma)$, defined in a conic set $\Gamma\subseteq \R^n\times(\R^N\setminus 0)$, there exists a unique, up to equivalence, almost analytic extension $\Tilde{a}\in S^m_{\textup{cl}}(\widetilde{\Gamma})$. Here, and in the rest of the document, $\widetilde{\Gamma}$ denotes a complexification of the real set $\Gamma$.

\begin{defi}
Let $\Omega\subseteq \C^n$ be an open set and $M\subseteq \Omega$ a closed submanifold of real dimension $2k$. We say that $M$ is an almost analytic manifold if for every real point $z_0\in M$, one can find an open neighborhood $\mathcal{O}$ of $z_0$ in $\Omega$ and almost analytic functions $f_{k+1},\dots,f_n$ such that, in $\mathcal{O}$,
\begin{itemize}
    \item $M$ is defined by $f_{k+1}=\dots=f_n=0$,
    \item and the differentials $\partial f_{k+1}(z),\dots,\partial f_n (z)$ are linearly independent over $\C.$
\end{itemize}
\end{defi}

The following theorem gives some useful description of almost analytic manifolds. 
\begin{thm}\thlabel{AAMequiv}
Let $\Omega\subseteq \C^n$ be an open set and $M\subseteq \Omega$ an almost analytic manifold. Then, for every real point $z_0\in M$, one can find a neighborhood $\mathcal{O}=\mathcal{O}'\times \mathcal{O}''\subseteq \C^k\times\C^{n-k}$ of $z_0$ in $M$ and an almost analytic function $h$ on $\mathcal{O}'$ such that, for all $z=(z',z'')\in \mathcal{O},$ $z''=h(z').$
\end{thm}
We also need the notion of equivalent almost analytic manifolds.  This is done \emph{locally}, by considering only neighborhoods of real points. 
\begin{defi}\thlabel{EquivAAMprop}
Let $M_1,M_2\subseteq \Omega \subseteq \C^n$ be almost analytic manifolds of the same dimension with $M_{1\R}=M_{2\R}$. Let $h_1, h_2$ be the defining functions in \thref{AAMequiv}. Then, $M_1\sim M_2$ if any of the following equivalent conditions is satisfied: 
\begin{enumerate}
    \item For all compact set $K\Subset \Omega$ and $N\in\Z_+$ there is a constant $C_{N,K}>0$ such that,
    \[|h_1(x')-h_2(x')|\leq C_{N,K}|\Im h_2(x')|^N,\qquad (x',h_j(x'))\in K,\ x'\in\R^k.\]
    \item For all compact set $K\Subset \Omega$ and $N\in\Z_+$ there is a constant $C_{N,K}>0$ such that,
    \[|h_1(z')-h_2(z')|\leq C_{N,K}|(z',\Im h_2(z'))|^N,\qquad (z',h_j(z'))\in K.\]
\end{enumerate}
\end{defi}

We now consider a key result: a complex-valued version of the so called \emph{stationary phase formula}. Let $F(x,w)\in\smooth(\R^n\times\R^k)$ be defined in a neighborhood of $(0,0)$. Assume that $\Im F\geq 0$ with equality only at the origin, and that
\[\partial_xF(0,0)=0,\qquad \det(\partial^2_xF(0,0))\neq 0.\]
Let $\widetilde{F}(z,\omega)$, $z=x+iy$, $\omega\in\C^k$, be an almost analytic extension to a complex neighborhood of $(0,0)$. Then, it can be shown that the equation $\partial_z\widetilde{F}(z,\omega)=0$ defines an almost analytic manifold $M$ of the form $z=Z(w)$, such that there is exists $C>0$ such that
\begin{align*}
    \Im \widetilde{F}(z,w)\geq C |\Im z|^2,\qquad (z,w)\in M,\ w\in\R^k.
\end{align*}

\begin{thm}\textup{(\cite[Theorem 2.3]{melin1975fourier})}\thlabel{spf}
Let $F$ and $Z$ be as above. Then, there are neighborhoods of the origin $U\in\R^n$, $V\in\R^k$ and differential operators $C_{\nu,w}(D)$, which are $\smooth$ functions of $w\in V$, and have order at most $2\nu$, such that 
\begin{align}\label{stationary-phase}
    \int e^{itF(x,w)}u_t(x)\ dx\sim \sum_{\nu=0}^{\infty} t^{-\nu-n/2}e^{it\widetilde{F}(Z(w),w)}\left(C_{\nu,w}(D)\Tilde{u}_t\right)Z(w), \quad t\to +\infty. 
\end{align}
in $S^{-n/2}_{\textup{cl}}(V\times\R_+)$. Here $u_t(x)\in S^0(\R^n\times \R_+)$ is supported in $U\times\R_+$ and the function $(2\pi)^{-n/2}C_{0,w}$ is the branch of the square root of $ \LR{\det\frac{1}{i}\partial^2_zF(Z(w),w)}^{-1}$ which continuously deform into 1 under the homotopy
\[[0,1]\ni s\mapsto \frac{1}{i}(1-s)\partial^2_zF+s I\in GL(n,\C). \]
\end{thm}

\subsection{Positive Lagrangian manifolds and complex-valued phase functions}

Let $M$ be a real symplectic manifold of dimension $2n$, fix a point  $\rho_0\in M$ and consider a coordinate neighborhood $W\subseteq \R^{2n}$ of $\rho_0$. Assuming that $\Lambda\subseteq \widetilde{M}$ is an almost analytic manifold containing $\rho_0$, the goal is to extend the symplectic structure of $M$ to $\widetilde{M}$. This is done locally, so the manifold $\Lambda$ is identified with its local representative  $\widetilde{W}\subseteq \C^{2n}$. Note that, given symplectic coordinates $(x,\xi)$ near $\rho_0$ in $M$, we can have coordinates in $\Lambda$ by taking almost analytic extensions $(\Tilde{x},\Tilde{\xi})$ to $\widetilde{W}.$ 

\begin{defi}
The manifold $\Lambda$ is called positive (almost) Lagrangian if, near every real point $(x_0,\xi_0)$, it is equivalent to a manifold of the form $$\Tilde{\xi}=\frac{\partial h}{\partial \Tilde{x}}(\Tilde{x}),\quad \Tilde{x}\in\C^n,$$ where $h$ is an almost analytic function satisfying $\Im h\geq 0$ on $\R^n$, with equality at $x_0.$
\end{defi}

So far, we have no information on the symplectic form $\sigma$. In fact, the manifolds on which $\sigma$ vanish, represent a special case.

\begin{defi}
An almost analytic manifold $\Lambda\subseteq \widetilde{M}$, of real dimension $2n$, is called strictly positive Lagrangian if 
\begin{enumerate}
    \item $\Lambda_\R$ is a submanifold of $M$.
    \item $\sigma_\alpha\vert_{\Lambda_\alpha}\sim 0$ for all local representatives $\Lambda_\alpha$ and all local almost analytic extensions $\sigma_\alpha$  of $\sigma$.
    \item $i^{-1}\sigma(v,v)>0$ for all $v\in T_\rho(\Lambda)\setminus \LR{T_\rho(\Lambda_\R)}^\sim$, $\rho\in\Lambda_\R$.
\end{enumerate}
\end{defi}

In practice, we will consider $M=T^*(X)\setminus 0$, for some manifold $X$, and $\Lambda\subseteq \LR{T^*(X)\setminus 0}^\sim$  positive Lagrangian. Under this conditions, it can be shown that $\Lambda$ is of the form $\Tilde{x}=H(\Tilde{\xi}),$
where $H$ is positive homogeneous of degree $0$ with $\Im H(\xi)\leq 0$ for $\xi$ real.

\begin{defi}\thlabel{defi.phasefun}
A complex-valued function $\phi(x,\theta)$ , smooth in a conic set $V\subseteq\R^n\times\R^N\setminus0$ is called a non-degenerate phase function of positive type if $\Im\phi\geq 0$ and 
    \begin{itemize}
        \item $d\phi\neq 0$,
        \item $\phi$ is homogeneous of degree 1 in $\theta$,
        \item the differentials $\left\{d(\frac{\partial\phi}{\partial\theta_j})\right\}_{i=1}^N$ are linearly independent over $\C$ on $C_{\phi\R}=\left\{(x,\theta)\in V\colon \phi'_\theta=0  \right\}.$
    \end{itemize}
\end{defi}

Let $\Tilde{\phi}(\Tilde{x},\Tilde{\theta})$ be an almost analytic extension of $\phi$, defined in a conic neighborhood $W\subseteq\C^n\times(\C^N\setminus 0)$ of the point $(x_0,\theta_0)\in V$. Then the critical set 
\[C_{\Tilde{\phi}}=\left\{(\Tilde{x},\Tilde{\theta})\in W\colon \partial_{\Tilde{\theta}} \Tilde{\phi}(\Tilde{x},\Tilde{\theta})=0  \right\}\]
is a conic almost analytic manifold of dimension $2n$. The image $\Lambda_{\Tilde{\phi}}$ of $C_{\Tilde{\phi}}$ under the map 
\begin{align}\label{map-CLambda}
(\Tilde{x},\Tilde{\theta})\mapsto \LR{\Tilde{x},\partial_{\Tilde{x}}\Tilde{\phi}(\Tilde{x},\Tilde{\theta})}
\end{align}

is locally, near $\rho_0=(x_0,\phi'_x(x_0,\theta_0))$, a conic positive Lagrangian manifold of dimension $2n$.  Moreover the image of $C_{\Tilde{\phi}\R}$ is precisely $\Lambda_{\Tilde{\phi}\R}$. Since any other choice of almost analytic extension $\Tilde{\phi}$ defines an equivalent almost analytic manifold, we write $\Lambda_\phi$. The reader should keep in mind that this represents an equivalent class of almost analytic manifolds, thus we could choose a different representative at any given moment. We refer to this process as a \emph{re-parametrization}, because it amounts to finding a new regular phase function equivalent to $\phi$. In fact, a positive Lagrangian manifold $\Lambda\subseteq \LR{T^*(X)\setminus 0}^\sim$ can always be parametrize by a non-degenerate phase function of the form \[\phi(x,\xi)=x\cdot\xi-g(\xi),\] for some almost analytic function $g$ which homogeneous of degree 1 and satisfy $\Im g(\xi)\leq 0$, for $\xi$ real.

It is important to mention that $\Lambda_{\Tilde{\phi}\R}\subseteq\R^n\times(\R^n\setminus 0)$ does not need to be a closed Lagrangian manifold in the usual sense. In some cases, it may not even be a manifold. This highlights the advantages of allowing complex-valued phase functions over the usual real-valued theory.

\subsection{Lagrangian distributions and their principal symbol}\label{chap:principal_symbol}

Let $V\subseteq\R^n\times\R^N\setminus0$ be a conic open set, $\phi\in\smooth(V)$ be a non-degenerate phase function, and $a\in S^m_{\textup{cl}}(\R^n\times\R^N\setminus0)$ be supported in a closed conic subset of $V$. We can formally define the distribution $A=I(\phi,a)\in\mathcal{D}'(\R^n)$ by 
\begin{align}\label{FD}
   I(\phi,a)=\int e^{i\phi(x,\theta)}a(x,\theta)\ d\theta.
\end{align}

Note that the contribution of $\Im\phi$ to the integral is a exponentially decreasing factor, which does not contribute to the singularities of the distribution. Thus we get, directly from the real case,  a result about the wave front set of $A$: 
\begin{align}
    \wf{A}\subseteq \left\{(x,\phi'_x(x,\theta))\colon (x,\theta)\in\operatorname{supp}(a)\cap C_{\phi\R}\right\}\subseteq \Lambda_{\phi\R}.
\end{align}

Let $X$ be a $\smooth$ paracompact manifold of dimension $n$ and denote by $\mathcal{D}'(X;\Omega^{\frac{1}{2}})$ be the space of $1/2$-densities in $X$. From this point forward, unless stated otherwise, $\Lambda\subseteq \LR{T^*(X)\setminus 0}^\sim$ denotes a positive Lagrangian manifold. 

\begin{defi}[Lagrangian distributions]\thlabel{FID}
We say that a distribution $A\in \mathcal{D}'(X;\Omega^{\frac{1}{2}})$ belongs to $I_{\textup{cl}}^m(X,\Lambda;\Omega^{\frac{1}{2}})$, if $\wf{A}\subseteq\Lambda_\R$ and there are functions $\phi\in \smooth(\R^n\times\R^N)$ and $a\in S^{m+(n-2N)/4}_{\textup{cl}}(\R^n\times\R^N)$ such that 
\begin{itemize}
    \item for every $\rho_0=(x_0,\xi_0)\in\Lambda_\R$ and every choice of local coordinates, $A$ is microlocally of the form $I(\phi,a)$ near $\rho_0$,
    \item  $\Lambda$ is parametrized by $\phi$, that is $\Lambda\sim\Lambda_{\phi}$ near $\rho_0$,
    \item $\supp {a}$ is contained in a small conic neighborhood of $(x_0,\theta_0)\in C_{\phi\R}.$
\end{itemize}
\end{defi}

For the global theory to be complete, one would have to define the principal symbol of $A\in I^m(X,\Lambda;\Omega^{\frac{1}{2}})$ invariantly. In analogy with the real case, that would be as a section of the tensor product of the bundle of $1/2$-densities in $\Lambda$ and the Maslov line bundle. But, it turns out that it is impossible to replicate this construction for complex manifolds in a way that is invariant under coordinate changes. To avoid this difficulty, \citeauthor{melin1975fourier} \cite{melin1975fourier} introduce \emph{admissible coordinates} and define a \emph{virtual} line bundle over $\Lambda$.  

This is an intricate  construction, so we first introduce the linear situation. Let  $\widetilde{M}$ be the complex extension of a real symplectic vector space $M$  of dimension $2n$, and denote by $\mathcal{L}^-$ the set of negative definite Lagrangian planes in $M$. Let $F\subseteq M$ be a fixed real Lagrangian plane and $\widetilde{F}$ its complexification. We denote by $B(F)$ the set of all real bases of $F$. 

\begin{defi}[Admissible basis]\thlabel{AdmissibleBasis}
Let $N\subseteq \widetilde{M}$ be a positive semi-definite Lagrangian plane. A basis $e=\{e_1,\dots,e_n\}$ of $N$ is said to be admissible if there exist a basis $f=\{f_1,\dots,f_n\}$ of $F$ and a plane $L\in\mathcal{L}^-$ such that, for each $j$, $e_j$ is the projection of $f_j$ along $L$. We write \[e=E(f,L)=E_{N}(f,L), \quad (f,L)\in B(F)\times \mathcal{L}^- \] and denote by $\mathscr{B}(N)$ the set of all admissible bases for $N$.
\end{defi}

\begin{prop}\thlabel{AdmBasTrans}
The set $\mathscr{B}(N)$ is the union of two disjoint arcwise-connected subsets. Two admissible bases $e=E(f,L),\ e'=E(f',L')$ belong to the same set if and only if $f,f'\in B(F)$ have the same orientation. Moreover, there exists a unique function $s=s_{N}: \mathscr{B}(N)\times \mathscr{B}(N)\to \C\setminus 0$ with the following properties
\begin{enumerate}
   \item For all compact set $K\subseteq B(F)\times \mathcal{L}^-$, $s(e,e')$ restricted to $E_{N}(K)\times E_{N}(K)$  is a continuous function of $e,e'$ and $N$.
    \item If $e,e',e''\in \mathscr{B}(N)$, then $s(e,e')s(e',e'')=s(e,e'').$
    \item If $e,e'$ have the same $L\in\mathcal{L}^-$, then $s(e,e')>0$.
     \item  $s^2(e,e')=\pm e/e'$ with the plus sing precisely when $f,f'$ have the same orientation. Here, 
    \[e/e'=e_1\wedge \dots \wedge e_n/ e'_1\wedge \dots \wedge e'_n, \quad e=E(f,L),e'=E(f',L')\in \mathscr{B}(N).\]
\end{enumerate}
\end{prop}

Consider now a positive Lagrangian manifold $\Lambda\subseteq \LR{T^*(X)\setminus 0}^\sim$ and $\rho\in\Lambda_\R$. Then, $M=T_\rho(T^*X)$ and $N=T_\rho(\Lambda)$ satisfy the conditions of \thref{AdmissibleBasis}. Namely, $T_\rho(T^*X)$ is a symplectic vector space and $T_\rho(\Lambda)\subset T_\rho(T^*X)$ is a positive semi-definite Lagrangian plane. Taking $F\subseteq M$ as the tangent space to the fiber, we can define $\mathscr{B}(T_\rho(\Lambda))$ as above. 

Keeping in mind the previous notation, we now define admissible coordinate systems on $\Lambda$. Denote by $E_{T_\rho(\Lambda)}(S)$ the set of all admissible bases $e=E(f,L)$ of $T_\rho(\Lambda)$ with 
$(f,L)\in S\subseteq B(F)\times \mathcal{L}^-.$

\begin{defi}[Admissible coordinate systems]
Let $\lambda=\{\lambda_1,\dots,\lambda_n\}$ be almost analytic functions on $\Lambda$, defined in some complex neighborhood $\mathcal{U}^\lambda$ of a real point. We say that $\lambda_1,\dots,\lambda_n$ are admissible coordinates on $\Lambda$ if
\begin{enumerate}
    \item The differentials $d\lambda_1,\dots,d\lambda_n$ are linearly independent over $\C$ at real points.
    \item $\delta\lambda=\{\delta\lambda_1,\dots,\delta\lambda_n\}$ belongs locally to $E_{T_\rho(\Lambda)}(K)$ with $\rho\in \mathcal{U}^\lambda\cap \Lambda_\R$, for some compact set  $K\subseteq B(F)\times \mathcal{L}^-$. Here $\delta\lambda$ is the dual basis of $d\lambda$ in $T_\rho(\Lambda)^*$.
\end{enumerate}
We refer to the neighborhoods $\mathcal{U}^\lambda$ as admissible coordinate systems.
\end{defi}

One can show that it is always possible to find admissible coordinates locally, however, this construction is not unique. It is precisely this property what allowed \citeauthor{melin1975fourier} to define the almost analytic version of the Maslov line bundle. 
All we need is a way to define the transition functions between two coordinate systems $\mathcal{U}^\lambda$ and $\mathcal{U}^\mu$.  Thanks to \thref{AdmBasTrans}, we know that 
\begin{itemize}
    \item $s(\delta\lambda,\delta\mu)$ is continuous in $\mathcal{U}^\lambda\cap \mathcal{U}^\mu\cap \Lambda_\R$.
    \item $s^2=\pm \displaystyle\frac{d\mu}{d\lambda}=\pm \det \left[\LR{\frac{\partial\mu_j}{\partial\lambda_k}}_{j,k}\right]$, where $\partial\mu_j/\partial\lambda_k$ is defined by
    \[d\mu_j=\sum_{k}\LR{\frac{\partial\mu_j}{\partial\lambda_k}}d\lambda_k+\sum_{k}\LR{\frac{\partial\mu_j}{\partial\overline{\lambda}_k}}d\overline{\lambda}_k.\]
\end{itemize}

Consider now an almost analytic extension $\textbf{S}$ of $s(\delta\lambda,\delta\mu)$, defined in a small complex neighborhood of $\mathcal{U}^\lambda\cap \mathcal{U}^\mu\cap \Lambda_\R$ in $\Lambda$. Thanks to the previous properties, $\textbf{S}$ can be chosen to satisfy 
\begin{align}
    \LR{\textbf{S}_{\lambda,\mu}}^2 &\sim \pm \displaystyle\frac{d\mu}{d\lambda},\label{trnasition-determinant}\\
    \textbf{S}_{\lambda,\lambda}\sim 1,& \quad  \textbf{S}_{\lambda,\mu}\textbf{S}_{\mu,\omega}\sim \textbf{S}_{\lambda,\omega}.
\end{align}
Additionally, the functions $\textbf{S}_{\lambda,\mu}$ are continuous under small perturbations of $\lambda,\ \mu$ for which $\delta\lambda,\ \delta\mu$ stay in the same component of $E_{T_\rho(\Lambda)}(K)$. For all of these, the functions $\textbf{S}_{\lambda,\mu}$ are the ideal choice of transition functions in the new almost analytic Maslov line bundle.

\begin{defi}
 The bundle $\mathscr{L}\to\Lambda$ is defined as the family of admissible coordinate systems $\mathcal{U}^\lambda$ on $\Lambda$ with transition functions $\textbf{S}_{\lambda,\mu}$. A section $f\in\Gamma(\Lambda,\mathscr{L})$ is an almost analytic function on $\Lambda$ such that, the restriction to each $\mathcal{U}^\lambda$ satisfy  $$f_{\lambda}\sim \textbf{S}_{\lambda,\mu} f_{\mu}.$$ The space of homogeneous section of degree $m$ is denote by $\Gamma^m(\Lambda,\mathscr{L})$.
\end{defi}

A particularly important homogeneous section is the one determined by a phase function $\phi$. The following lemma is a crucial step in defining the principal symbol of a distribution in $I^m_{cl}(X,\Lambda;\Omega^{\frac{1}{2}})$.

\begin{lemma}\thlabel{Dphi}
Let $\phi(x,\theta)$ be a non-degenerate phase function that parametrizing $\Lambda$ near $\rho_0\in\Lambda_\R$. Then, there is a section $\sqrt{d\phi}\in \Gamma^{N/2}(\Lambda,\mathscr{L})$, defined by
        \begin{align}\label{dphi}
            (\sqrt{d\phi})_{\tau}\sim \left[\det\frac{1}{i} 
                \left(\begin{array}{cc}
                \Tilde{\phi}''_{xx}-\Tilde{\psi}''_{xx} & \Tilde{\phi}''_{x\theta}  \\
                \Tilde{\phi}''_{\theta x}& \Tilde{\phi}''_{\theta\theta}
                \end{array}\right)\right]^{-1/2},
        \end{align}
 where $\tau=\widetilde{\xi}- \Tilde{\psi}'_{\widetilde{x}}$ is an admissible coordinate system on $\Lambda$, and $\psi\in\smooth(\R^n)$ satisfy $\Tilde{\psi}''_{xx}<0$. The branch of the square root is chosen as in Theorem \ref{spf}.
\end{lemma}

We now recall the definition of Fourier integral operator with complex phase. Let $X,\ Y$ be paracompact $\smooth$ manifolds of dimension $n_X$, $n_Y$ respectively. As with the real case, if $C\subseteq \LR{T^*X\setminus 0}^\sim\times \LR{T^*Y\setminus 0}^{\sim}$ is an arbitrary submanifold, we denote by $C'$ the manifold
\[\left\{(x,y,\xi,-\eta)\colon (x,\xi,y,\eta)\in C \right\}\subseteq \LR{T^*(X\times Y)\setminus 0}^\sim.\]

\begin{defi}
The submanifold $C\subseteq \LR{\LR{T^*X\setminus 0}\times \LR{T^*Y\setminus 0}}^{\sim}$ is a (strictly) positive canonical relation, if $C'\subseteq \LR{T^*(X\times Y)\setminus 0}^\sim$ is a closed conic (strictly) positive Lagrangian manifold and $C_\R\in \LR{T^*X\setminus 0}\times \LR{T^*Y\setminus 0}$.  
\end{defi}
Here \emph{closed} means that $C'_\R$ is a closed set of $T^*(X\times Y)\setminus 0$. Finally, we recall the definition of Fourier integral operator with complex phase.

\begin{defi}\thlabel{C-FIO}
 An operator $A:\smooth_0(Y;\Omega^{\frac{1}{2}})\to\mathcal{D}'(X;\Omega^{\frac{1}{2}})$ is called a Fourier integral operator with complex phase if its distributional kernel $K_A$ belongs to the class of Lagrangian distributions $I_{cl}^m(X\times Y,\Lambda;\Omega^{\frac{1}{2}})$. Here $\Lambda\subseteq  \LR{T^*(X\times Y)\setminus 0}^\sim$ is a closed conic positive Lagrangian manifold satisfying $C'=\Lambda$ for some $C\subseteq \LR{\LR{T^*X\setminus 0}\times \LR{T^*Y\setminus 0}}^{\sim}$. We write $A\in I_{\textup{cl}}^m(X\times Y,C;\Omega^{\frac{1}{2}})$.
\end{defi}

It follows form the real case that, whenever $C$ is a canonical relation, the operator $A$ maps \[\smooth_0(Y;\Omega^{\frac{1}{2}})\to\smooth(X;\Omega^{\frac{1}{2}}),\] moreover, it can be extended to a continuous operator from $\mathscr{E}'(Y;\Omega^{\frac{1}{2}})$ to $\mathcal{D}'(X;\Omega^{\frac{1}{2}})$.

\section{The principal symbol map}The following construction shows how the complex-valued stationary phase formula can be used to give an explicit description of the principal symbol of a Lagrangian distribution with complex phase. In particular, given a distribution $A\in I_{cl}^m(X,\Lambda;\Omega^{\frac{1}{2}})$, which is microlocally of the form $I(\phi,a)$ near some real point $\rho$,  we are able to see the relation between the amplitude $a$ and the action of the map $\mathcal{P}$ in \cite{melin1975fourier}. 

We follow the ideas of \citeauthor{duistermaat1973fourier} for the real case (see \cite{duistermaat1973fourier}) and adapt them to the complex domain. Namely, we use the asymptotic expansion in \thref{spf} to provide a local description of the principals symbol. Later, we show that this description corresponds to the definition provided in \cite{melin1975fourier}. 

\begin{lemma}
Let $\phi\in\smooth(\R^n\times (\R^N\setminus 0))$ be a non-degenerate phase function and $(x_0,\theta_0)\in C_{\phi\R}$ fixed. If $\psi\in\smooth(\R^n)$ is real valued and 
\[\psi(x_0)=0,\quad \psi'_x(x_0)=\phi'_x(x_0,\theta_0),\quad \psi''_{xx}<0.\] Then, the function $F(x,\theta)=\phi(x,\theta)-\psi(x)$ satisfies the assumptions of \thref{spf} around $(x_0,\theta_0)$.
\end{lemma}
\begin{proof}
It is clear that $F(x,\theta)$ is smooth with $\Im F\geq 0$. From the definition of $C_{\phi\R}$, one  sees that $\Im F(x_0,\theta_0)=0$. Thus, we only need to show that
\[\partial_x F(x_0,\theta_0)=0,\qquad \det(\partial^2_x F(x_0,\theta_0))\neq 0.\]
A short computation shows that 
\begin{align*}
    \partial_x F(x_0,\theta_0)= \phi'_x(x_0,\theta_0)-\psi'_x(x_0)=0, \qquad
    \partial^2_x F = 
                \left(\begin{array}{cc}
                \phi''_{xx}-\psi''_{xx} & \phi''_{x\theta}  \\
                \phi''_{\theta x}& \phi''_{\theta\theta}
                \end{array}\right).
\end{align*}
This matrix is the same one in the definition of $\sqrt{d\phi}$ (\thref{Dphi}). Then, the result follows by the same arguments, see \cite{melin1975fourier}, Theorem 6.4. 
\end{proof}

Let $X$ be a smooth manifold of dimension $n$ and consider a distribution $A\in I_{cl}^m(X,\Lambda;\Omega^{\frac{1}{2}})$, which is of the form $I(\phi,a)$ microlocally near $\rho_0=(x_0,\xi_0)$, $\xi_0=(x_0,\phi'_x(x_0,\theta_0))$ with $\phi$ a non-degenerate phase function and $a(x,\theta)\in S^{m+(n-2N)/4}_{cl}(\R^n\times (\R^N\setminus 0))$. Let $u\in C^{\infty}_c$ be supported in a neighborhood of $(x_0,\theta_0)$ and $\psi$ be as in the previous lemma.  We want to understand the asymptotic behaviour of  $I:=(I(\phi,a),v)_{L^2}$, as $t\to \infty$, with $v(x)=e^{-it\psi(x)}u(x)$. By definition, 
\begin{align*}
    (I(\phi,a),v)_{L^2}=\int e^{i\phi(x,\theta)}a(x,\theta)e^{-it\psi(x)}u(x)\ dxd\theta= \int e^{i[\phi(x,\theta)-t\psi(x)]}a(x,\theta)u(x)\ dxd\theta .
\end{align*}
After the change of variables $\theta=t\eta$, we obtain
\begin{align*}
    I=t^N\int e^{it F(x,\eta)}u_t(x,\eta)\ d\eta\ dx,
\end{align*}
with $F=\phi-\psi$ and $u_t=a(x,t\eta)u(x)$. Thanks to the previous lemma, we know that the complex-valued stationary formula applies here. Thus we get from \thref{spf},
\[e^{-it\widetilde{F}(Z(\widetilde{\eta}),\widetilde{\theta})}I\sim \sum_{\nu=0}^{\infty} t^{-\nu-(n+N)/2}\left(C_{\nu,\eta}(D)\widetilde{u_t}\right)Z(\widetilde{\eta}),\]
where $x=Z(\widetilde{\eta})$ is the almost analytic manifold described by $\partial_x\widetilde{F}(\widetilde{x},\widetilde{\eta})=0.$ We can now describe the principal symbol of $A$ as the map that assigns to each $\psi$ the top order term of the asymptotic expansion of $I$. Noting that $C_{0,w}(D)=(2\pi)^{\frac{n+N}{2}}\sqrt{d\phi}$, we can write this map explicitly 
\begin{align}\label{PSmap2}
    \mathcal{T}_A:\psi \mapsto (2\pi)^{\frac{n+N}{2}}\widetilde{a_0}(Z(\widetilde{\eta}),\widetilde{\eta})\widetilde{u}(Z(\widetilde{\eta}))\sqrt{d\phi},
\end{align}
where $\sqrt{d\phi}\in\Gamma^{N/2}(\Lambda,\mathscr{L})$ and  $a_0$, the highest order term in the asymptotic sum of $a$, is and homogeneous function of degree $m+(n-2N)/4$ in $\eta$. The final step is to relate this expression with the formulation of the principal symbol in \cite{melin1975fourier}. Or, equivalent, we prove our main result \thref{PS2}, which is stated again bellow.

Let $s\in\Gamma^{d_1}(\mathscr{L},\Lambda)$ and $b$ an homogeneous function of degree $d_2$ in $\theta$. In the rest of the document, we denote by  $S^{(d)}(\Lambda,\mathscr{L})$, with $d=d_1+d_2$, the space of almost analytic functions $f\sim bs$. 

\begin{thm}
 Let $A=I(\phi,a)$, with $\phi(x,\theta)$ a non-degenerate phase function and $a(x,\theta)\in S^{m+(n-2N)/4}_{cl}$. Then, the principal symbol of $A\in I^m_{cl}(X,\Lambda,\Omega^{\frac{1}{2}})$ is equivalent to the homogeneous section 
\begin{align}\label{ps2}
    \sigma^m(A)\sim\widetilde{a_0} \sqrt{d\phi}\in S^{(m+n/4)}(\Lambda,\mathscr{L}),
\end{align}
where $a_0$ is the top order term of the asymptotic expansion of $a$.
\end{thm}

\begin{proof}
Let be $\mathcal{P}$ be the bijection in \cite{melin1975fourier}. The result would follow after showing that, $\mathcal{T}_A(\psi)\sim\mathcal{P}^{-1}(A)$ as sections of $\mathscr{L}$. But first, we need to verify that $\mathcal{T}_A(\psi)\in \Gamma^{m+n/4}(\Lambda,\mathscr{L})$.

It is not hard to see that almost analytic homogeneous function define homogeneous sections in $\mathscr{L}$.  Indeed, let $g$  be an almost analytic function in $\Lambda$, homogeneous of degree $m$ in $\xi$. We know that for each $\rho\in\Lambda_\R$, it is possible to find and admissible coordinate system $\mathcal{U}^\lambda$, locally near $\rho$. Then, we can define a function $f$ such that $f_\lambda=g\vert_{ \mathcal{U}^\lambda}$.  Clearly this defines a section of $\mathscr{L}$.  Now, fix $\rho=(x,\xi)\in\Lambda_\R$ and $t\in\R_+$. Then, near $\textbf{t}\rho$, points in $\Lambda$ are of the form $(x,t\xi)$ . Let $\lambda=\{\lambda_1,\dots,\lambda_n\}$  and $\mu=\LR{\textbf{t}^{-1}}^*\lambda$ be admissible coordinates near $\rho$ and $\textbf{t}\rho$, respectively. Explicitly, for each $j=1,\dots,n$, 
\[\mu_j(x,t\xi)=\LR{\lambda_j\circ\textbf{t}^{-1}}(x,t\xi)=\lambda_j(x,\xi).\] 
Then, near $\rho$, we have $\LR{\textbf{t}^{-1}}^*(f_\lambda)(x,\xi)\sim f\LR{x,\frac{1}{t}\xi}\sim t^{-n/2}f_\lambda(x,\xi)$ and
\[f_{\LR{\textbf{t}^{-1}}^*\lambda}\sim f(x,t\xi)=t^m f(x,\xi)=t^{n/2}t^{-n/2}t^m f(x,\xi)\sim t^{m-n/2}f_\lambda(x,\xi)\sim t^m \LR{\textbf{t}^{-1}}^*\LR{f_\lambda}.\] 
Which, by definition, means that $f\in \Gamma^m(\Lambda,\mathscr{L})$. This, together with the local identification of $C_\phi$ and $\Lambda\sim\Lambda_\phi$, allow us to interpret $\widetilde{a_0}$ and $\widetilde{u}$ as elements of $\Gamma^{m+(n-2N)/4}(\Lambda,\mathscr{L})$ and $\Gamma^{0}(\Lambda,\mathscr{L})$, respectively. Then, it follows that the right hand side of \eqref{PSmap2} defines an homogeneous section of degree $m+n/4$. 

Recall that a distribution $A\in I^m_{\textup{cl}}$ is locally given by an oscillatory integral 
\[A=I(\phi,a)=\int e^{i\phi(x,\eta)}a(x,\eta)\ d\eta,\]
where $a\in S^{m+(n-2N)/4}_{cl}$. We can assume, without loss of generality, that $\mathcal{P}(s)=I(\phi,\overline{b_0})$ where $b_0$ is the top-order term of the asymptotic sum of some $b\in S^{m+(n-2N)/4}_{cl}$ and $s$ is some section in $\Gamma^{N/2}(\Lambda,\mathscr{L})$. Thus, it remains to show that $\mathcal{P}\circ \mathcal{T}_A$ is the identity in the set of oscillatory integrals $I(\phi,b_0)$, with $b_0$ as above. 

The action of $\mathcal{P}$, tells us that if $s\sim b\sqrt{d\phi}\in \Gamma^{m+n/4}(\Lambda,\mathscr{L}),$ then $\mathcal{P}(s)=I(\phi,\overline{b})$ with $\overline{b}$ an extension of $b$ to $\C^n\times \C^N$. Then, taking $s=\mathcal{T}_A(\psi)\sim \widetilde{a_0}\widetilde{u}\sqrt{d\phi}$, we have $\mathcal{P}(s)=I(\phi,\overline{b})$ with $b=\widetilde{a_0}\widetilde{u}$. Since this is valid for any $u\in C^{\infty}_c$ and any almost analytic extensions, we conclude that \[\mathcal{P}(s)\sim I(\phi,\widetilde{a_0}).\]

Or, equivalently, $\mathcal{P}(\mathcal{T}_A(\psi)\sim I(\phi,\widetilde{a_0})$, where $a_0$ is the highest order term of the amplitude $a$ in the local representation $A=I(\phi,a)$, which concludes the proof. 
\end{proof}

It is now easy to see that the principal symbol map fits into a short exact sequence, equivalent to the one in the real-valued case. First note that $\sigma^m$ maps
\begin{align}\label{PSM-complex}
\sigma^{m}:I^m_{\textup{cl}}(X,\Lambda;\Omega^{\frac{1}{2}})\to S^{(m+n/4)}(\Lambda,\mathscr{L}).
\end{align} 
Moreover, the fact that $\mathcal{P}$ is bijective, implies that $\sigma^{m}$ is surjective. Thus, similarly to the real-valued case, the map $\sigma^{m}$ fits into a short exact sequence
\begin{align*}
0\to I^{m-1}_{\textup{cl}}(X,\Lambda;\Omega^{\frac{1}{2}}) \to I^m_{\textup{cl}}(X,\Lambda;\Omega^{\frac{1}{2}})\xrightarrow{\sigma^{m}}S^{(m+n/4)}(\Lambda,\mathscr{L}) \to 0.
\end{align*}

\section{Clean composition}In \cite{melin1975fourier}, the case of \emph{transverse composition} was considered. We wish to relax this condition, so we consider the case of \emph{clean composition}. To do so, we first need to introduce a more general type of phase function that we were not able to find in the existing literature, even though it is a natural generalization of the non-degenerate case. 
\begin{defi}
 A complex-valued function $\phi(x,\theta)$, smooth in a conic set $V\subset\R^n\times\R^N\setminus0$,  is called clean phase function of positive type if $\Im\phi\geq 0$ and 
    \begin{itemize}
        \item $d\phi\neq 0$,
        \item $\phi$ is homogeneous of  of degree 1 in $\theta$,
        \item there exist $M\leq N$, such that $M$ of the differentials $\left\{d(\frac{\partial\phi}{\partial\theta_j})\right\}_{j=1}^N$ are linearly independent over $\C$ on 
        \[C_{\phi\R}=\left\{(x,\theta)\in V\colon \phi'_\theta=0  \right\}.\]
        The number $e=N-M$ is called the excess of $\phi$.
    \end{itemize}
\end{defi}

It is possible to organize the variables so that we can write $\theta=(\theta',\theta'')\in\R^M\times \R^e$, where  the differentials $\left\{d(\partial\phi/\partial\theta'_j)\right\}$ are the ones satisfying the definition. As usual, we denote by $\Lambda_\phi$ the manifold 
\[\left\{\LR{\Tilde{x},\partial_{\Tilde{x}}\Tilde{\phi}(\Tilde{x},\Tilde{\theta})}\in \C^n\times\C^N\setminus 0 \colon (\Tilde{x},\Tilde{\theta})\in C_{\Tilde{\phi}}  \right\},\]
for some almost analytic extension $\Tilde{\phi}$ of $\phi$ to a complex extension of $V$. 
One can easily verify that $\Lambda_\phi$ is a positive Lagrangian manifold of real dimension $2n$. As such, it can be parameterized by a non-degenerate phase function,
\begin{align}\label{reg-clean-equiv-phase}
    \widetilde{\psi}(\Tilde{x},\Tilde{\xi})=\Tilde{x}\Tilde{\xi}-g(\Tilde{\xi}),\quad (\Tilde{x},\Tilde{\xi})\in \C^n\times (\C^n\setminus 0)
\end{align}
for some almost analytic function $g$ with $\Im g\leq 0$ at $\xi\in\R^n\setminus 0.$ Denoting by $\psi$ the restriction of $\widetilde{\psi}$ to the real domain, we see that $\Lambda_\phi$ and $\Lambda_\psi$ are equivalent in the sense of almost analytic manifolds. This equivalence allows us to associate distributions in $I^m_{\textup{cl}}$  with a microlocal representation $I(\phi,a)$, where $\phi$ is a clean phase function instead of a non-degenerate one. 

 \begin{prop}\thlabel{CleanPhaseDist}
Let $\phi(x,\theta)\in\smooth(\R^n\times(\R^N\setminus0))$ be a clean phase function of excess $e$. Then, for $a\in S^{m+(n-2N)/4}(\R^n\times(\R^N\setminus0))$, the oscillatory integral $I(\phi,a)$ defines a Fourier distribution of order $m+e/2$.
\end{prop}
\begin{proof}
Fix a point $(x_0,\theta_0)\in\C_{\phi\R}$ and let $\xi_0=\phi'_x(x_0,\theta_0)$. We know that, near $(x_0,\xi_0)$, the almost analytic manifold $\Lambda\sim\Lambda_\phi$ is equivalent to a manifold $\Lambda_\psi$, with $\psi$ the non-degenerate phase function \eqref{reg-clean-equiv-phase}. We wish to follow the construction in \cite[Theorem 4.2]{melin1975fourier} to show that there exist an amplitude $b\in S^{m+(n-2N')/4}(\R^n\times(\R^{N'}\setminus0))$, $N'=n$, such that the oscillatory integrals $I(\phi,a)$ and $I(\psi,b)$  are microlocally equivalent near $(x_0,\xi_0)$. Which implies that $I(\phi,a)$ defines a Fourier distribution. 

The proof is based on the complex-valued stationary phase formula (\thref{spf}). There, it was possible to apply the result with respect to the variables $(x,\theta)$, because the phase functions were assumed to be non-degenerate. This is no longer true for a clean phase function $\phi$. Instead, we need to consider $\theta=(\theta',\theta'')\in(\R^M\times \R^e)\setminus0$, and
\begin{align}\label{clean-2int}
    I(\phi,a)\sim \int \left(\int e^{i\phi(x,\theta',\theta'')}a(x,\theta',\theta'')\ d\theta'\right)\ d\theta''.
\end{align}
Since the differentials  $\left\{d(\partial\phi/\partial\theta'_j)\right\}$ are linearly independent over $\C$ at real points, we can apply the stationary phase formula to the inner integral in \eqref{clean-2int}. The rest of the argument in \cite[Theorem 4.2]{melin1975fourier} applies without further modification. The desired conclusion follows after integrating out the excess variables $\theta''$.

Finally, note that applying the stationary formula in $e$ variables less, increases the order of the distribution by $e/2$. Now the asymptotic sum \eqref{stationary-phase} is in $S^{-(n+N)/2+e/2}$ instead of $S^{-(n+N)/2}$, as it was the case in \cite{melin1975fourier}.
\end{proof}

Since the construction that leads to our description of the principal symbol is also based on the stationary phase formula,  equation \eqref{ps2} applies only to the inner integral in \eqref{clean-2int}. In other words, the principal symbol of the distribution $A=I(\phi,a)$ above is be given by the integral, with respect to  $\theta''$ , of the principal symbol of the inner distribution. But, for this to be correctly defined, we first need to modify the definition of $\sqrt{d\phi}$.

\begin{lemma}\thlabel{Dphi-clean}
Let $\phi(x,\theta)$ be a clean phase function with excess $e$ that parameterizes $\Lambda$. Then, there is a section $\sqrt{d\phi}\in \Gamma^{(N-e)/2}(\mathscr{L},\Lambda)$, defined by
        \begin{align}\label{dphi-clean}
            (\sqrt{d\phi})_{\tau}\sim \left[\det\frac{1}{i} 
                \left(\begin{array}{cc}
                \Tilde{\phi}''_{xx}-\Tilde{\psi}''_{xx} & \Tilde{\phi}''_{x\theta'}  \\
                \Tilde{\phi}''_{\theta' x}& \Tilde{\phi}''_{\theta'\theta'}
                \end{array}\right)\right]^{-1/2},
        \end{align}
        where $\theta''$ are the excess variables in the splitting $\theta=(\theta',\theta'')$. Here $\tau$, $\psi$ and the branch of the square root are chosen as in \thref{Dphi}.
\end{lemma}
\begin{proof}
Note that, for $\theta''$ fixed, $\phi$ defines a non-degenerate phase function with respect to the variables $(x,\theta')$. Then, it follows from \thref{Dphi}, that $\sqrt{d\phi}$ defines a section of $\mathscr{L}$. Since the matrix in \eqref{dphi-clean} is now of dimension $(n+N-e)\times (n+N-e)$, we see that $\sqrt{d\phi}$ is homogeneous of degree $(N-e)/2$, which completes the proof. 
\end{proof}

With this new meaning for $\sqrt{d\phi}$, we can apply \thref{PS2} to the inner integral in equation \eqref{clean-2int}. It follows that $\widetilde{a_0} \sqrt{d\phi}$ defines an element of $S^{(m-e+n/4)}(\Lambda,\mathscr{L}),$ but, we still need to integrate out the excess variables $\theta''$. In principle, this integral may not be defined. Thus, similar to the real case, we restrict the domain of integration. 

Let $\pi:\Lambda_\phi\to\C^n$ be the projection $\pi(\Tilde{x},\Tilde{\xi})=\Tilde{\xi}.$ The composition of $\pi$ with the map \eqref{map-CLambda} defines a fiber bundle over $\Lambda$ with fiber
\[C_{\Tilde{\xi}}=\left\{(\Tilde{x},\Tilde{\theta})\colon \partial_{\Tilde{\theta}}\Tilde{\phi}(\Tilde{x},\Tilde{\theta})=0,\ \partial_{\Tilde{x}}\Tilde{\phi}(\Tilde{x},\Tilde{\theta})=\Tilde{\xi} \right\}.\] 
The fiber $C_{\Tilde{\xi}}$ can be interpreted as an almost analytic manifold of dimension $2e$, if the differentials $\left\{d(\frac{\partial\phi}{\partial x_j})\right\}_{j=1}^n$ are linearly independent at real points.  In any case, we 
we can compute $\int_{C_{\Tilde{\xi}}}\widetilde{a_0} \sqrt{d\phi}\ d\theta''$ if we assume that $C_{\xi\R}$, the restriction of $C_{\Tilde{\xi}}$ to the real domain, is compact.

\begin{defi}\thlabel{ps-clean-rem}
 Let $\phi(x,\theta)\in\smooth(\R^n\times(\R^N\setminus 0))$ be a clean phase function with excess $e$, such that the set $C_{\xi\R}$ is compact, and $a(x,\theta)\in S^{m+(n-2N-2e)/4}_{\textup{cl}}$. Then, the principal symbol of $A\in I^{m}_{\textup{cl}}(X,\Lambda;\Omega^{\frac{1}{2}})$ is 
\begin{align}\label{ps-clean}
\sigma^{m+e/2}(A)=\int_{C_{\xi\R}}\widetilde{a_0} \sqrt{d\phi}\ d\theta''\in S^{(m-e+n/4)}(\Lambda,\mathscr{L}).
\end{align}
where $A=I(\phi,a)$ and $a_0$ denotes the top order term of the asymptotic expansion of $a$.
\end{defi}

Now let $X,\ Y,\ Z$ be paracompact $\smooth$ manifolds of dimension $n_X$, $n_Y$, $n_Z$ respectively. Suppose that 
\[A_1\in I^{m_1}_{cl}(X\times Y, \Lambda_1;\Omega^{\frac{1}{2}})\quad \text{and} \quad A_2\in I^{m_2}_{cl}(Y\times Z, \Lambda_2;\Omega^{\frac{1}{2}}),\]
are properly supported operators, where $C_j=\Lambda'_j$ are positive canonical relations. Denote by $\Delta_Y$ the subspace \[\operatorname{diag}(Y)=\{(y,y)\in Y\times Y\}.\]
Set $D=T^*X\times \Delta_{T^*Y}\times T^*Z$ and let $\widetilde{D}$ be its almost analytic complexification.

We consider the following condition
 
\begin{ass}\thlabel{CleanInt} Suppose that:
\begin{enumerate}
    \item The intersection $(C_{1\R}\times C_{2\R})\cap D$ is clean with excess $e$.
    \item The natural projection $(C_{1\R}\times C_{2\R})\to (T^*X\setminus 0)\times (T^*Z\setminus 0)$ is injective and proper.
\end{enumerate}
\end{ass}

Note that, whenever the excess is equal to zero, the intersection is actually transverse. In other words, after showing that under \thref{CleanInt}, the composition of Lagrangian distributions stays in the class; the composition theorem in \cite{melin1975fourier} follows as a particular case. 
 
The following results follow form arguments similar to those in \cite{melin1975fourier}, so we only present the parts where the proofs are different. 

 \begin{prop}\thlabel{CleanComp}
Let  $C_1\in (T^*(X\times Y)\setminus 0)^{\sim}$, $C_2\in(T^*(Y\times Z\setminus 0))^{\sim}$ be positive canonical relations satisfying \thref{CleanInt}. Then, there exist a manifold $(C_1\circ C_2)'$, parameterized by a clean phase function $\Phi$, such that $(C_1\circ C_2)_{\R}=(C_{1\R}\circ C_{2\R})$. The excess of $\Phi$ is equal to the excess of the intersection. 
 \end{prop}
 \begin{proof}
Since the $\Lambda_j=C'_j$, $j=1,2$, are almost Lagrangian manifolds, there are coordinates such that, in a neighborhood of real points $(x_0,\xi_0,y_0,-\eta_0)\in\Lambda_{1\R}=C_{1\R}'$ and $(y'_0,\eta'_0,z_0,\zeta_0)\in\Lambda_{2\R}=C_{2\R}'$, the manifolds are given by the vanishing of 
\[\Tilde{x}-\frac{\partial H_1}{\partial\Tilde{\xi}}(\Tilde{\xi},\Tilde{\eta}),\ \Tilde{y}+\frac{\partial H_1}{\partial\Tilde{\eta}}(\Tilde{\xi},\Tilde{\eta});\qquad \Tilde{y}'-\frac{\partial H_2}{\partial\Tilde{\eta}'}(\Tilde{\eta}',\Tilde{\zeta}),\ \Tilde{z}+\frac{\partial H_2}{\partial\Tilde{\zeta}}(\Tilde{\eta}',\Tilde{\zeta}). \]

The intersection $(C_1\times C_2)\cap \Tilde{D}$ is completely describe by the previous functions and 
\[\Tilde{y}=\Tilde{y}'\qquad \Tilde{\eta}=\Tilde{\eta}'.\]
Its tangent plane is given by the vanishing of the differentials of all these functions. Clean intersection means that $T(C_{1\R}\times C_{2\R})\cap TD$ is described by the equations 
\begin{equation}
    \begin{aligned}\label{defining differentials}
    &d\left(x-\frac{\partial H_1}{\partial\xi}(\xi,\eta)\right)=0, & d\left(y+\frac{\partial H_1}{\partial\eta}(\xi,\eta)\right)=0, && d(y-y')=0,\\
    &d\left(y'-\frac{\partial H_2}{\partial\eta'}(\eta',\zeta)\right)=0, & d\left(z+\frac{\partial H_2}{\partial\zeta}(\eta',\zeta)\right)=0, && d\left(\eta-\eta'\right)=0,
\end{aligned}
\end{equation}

and has dimension $n_X+n_Z+e$, where $e$ is the excess of the intersection. As in the transversal case, define $C_1\circ C_2$ as the manifold that satisfy $(C_1\circ C_2)_{\R}=(C_{1\R}\circ C_{2\R})$ where 
\begin{align*}
    C_{1\R}\circ C_{2\R}&=\left\{((x,y,\xi,\eta),(y',z,\eta',\zeta))\in C_{1\R}\times C_{2\R}\colon y=y',\ \eta+\eta'=0 \right\}
\end{align*}
can be identified with $(C_{1\R}\times C_{2\R})\cap D$. The main difference is that now  $\Lambda=(C_1\circ C_2)'$ is of dimension $n_X+n_Z+e$. Suppose now that $\Lambda_1$ and $\Lambda_2$ are parameterized  by the regular phase functions 
\[\phi_1(x,y,\xi,\eta)=x\cdot\xi-y\cdot\eta+H_1(\xi,\eta),\quad \phi_2(y,z,\eta,\zeta)=y\cdot\eta-z\cdot\zeta+H_2(\eta,\zeta).\]
The previous analysis shows that the function $$\Phi(x,z,\omega)=\phi_1(x,y,\xi,\eta)+\phi_2(y,z,\eta,\zeta),\qquad \omega=\omega(y,\xi,\eta,\zeta),$$ defines only a clean phase function with excess $e$, because the differentials $d(\partial \Phi/\partial \omega_j)$ are exactly those in  \eqref{defining differentials}. Then, the excess of the phase function $\Phi$ is $$\dim \Lambda-(n_X+n_Z)=e.$$

Finally, note that there is a one-to-one correspondence between $C_{1\R}\circ C_{2\R}$ and
\[ C_{\Phi\R}=\left\{((x,y,\xi,\eta),(y,z,\eta,\zeta))\in C_{\phi_1}\times C_{\phi_2}\colon \partial_y(\phi_1+\phi_2)=0 \right\},\]
so the manifold $\Lambda=(C_1\circ C_2)'$ can be parameterized by a the clean phase function $\Phi$.
 \end{proof}
 
We now present an extension of the composition theorem to the case of clean intersection. The proof is omitted because it follows the same arguments as Theorem 7.3 in \cite{melin1975fourier}, but using the canonical transformation given by \thref{CleanComp}. The order of the distribution follows from \thref{CleanPhaseDist}. 

To make our notation consistent with the one used in \cite{melin1975fourier}, from now on we put $\theta=(\xi,\eta)$ and $\sigma=(\eta,\zeta)$. Then, we consider $\Phi$ as a clean phase function depending on $(x,z,y,\theta,\sigma)$, that is
\begin{align}\label{newPhi}
\Phi(x,z,\omega)=\phi_1(x,y,\theta)+\phi_2(y,z,\sigma),\qquad \omega=\omega(y,\theta,\sigma).
\end{align}
We can consider, for instance, $\omega(y,\theta,\sigma)=\left(y(\theta^2+\sigma^2)^{1/2},\theta,\sigma\right)$.

\begin{thm}\thlabel{clean intersection}
Let $C_1\subset (T^*(X\times Y)\setminus 0)^\sim$, $C_2\subset (T^*(Y\times Z)\setminus 0)^\sim$ be positive canonical relations  satisfying \thref{CleanInt}. Suppose that $A_1\in I^{m_1}_{cl}(X\times Y, C_1';\Omega^{\frac{1}{2}})$ and $A_2\in I^{m_2}_{cl}(Y\times Z, C_2';\Omega^{\frac{1}{2}})$ are properly supported. Then $A_1\circ A_2\in I^{m}_{cl}(X\times Z, (C_1\circ C_2)';\Omega^{\frac{1}{2}})$, $m=m_1+m_2+e/2$, where $e$ is the excess of the intersection. 
\end{thm}

Assuming that $A_1$ and $A_2$ have local representations $I(\phi_1,a_1)$ and $I(\phi_2,a_2)$, respectively, one can see that $B=A_1\circ A_2$ is, modulo $\smooth$, locally given by $I(\Phi,b)$. Where $\Phi$ is the clean phase function in equation \eqref{newPhi}, and the amplitude $b\in  S_{cl}^{m+(n_X+n_Z-2N)/4}(X\times Z\times \R^N\setminus 0)$ is of the form 
\[b(x,z,\omega)=a_1(x,y,\theta)a_2(y,z,\sigma)(\theta^2+\sigma^2)^{-n_Y/2},\]
with $m=m_1+m_2$ and $N=N_1+N_2+n_Y$.
\begin{rem}
For the amplitude $b$, the estimates of the class $ S_{cl}^{m'}(X\times Z\times \R^N\setminus 0)$, are taken with respect to $(x,z,\omega)$, with $\omega$ the homogeneous function above.
\end{rem}

 It is then clear that we should be able to write the principal symbol of $B$ in terms of $\sigma(A_1)$ and $\sigma(A_2)$. The next theorem shows precisely this in the case of transverse composition.  

\begin{thm}\thlabel{PScomp}
Let $A_1\in I^{m_1}_{cl}(X\times Y, C_1';\Omega^{\frac{1}{2}})$ and $A_2\in I^{m_2}_{cl}(Y\times Z, C_2';\Omega^{\frac{1}{2}})$ satisfy \thref{CleanInt} with excess $e=0$, and denote by $B$ the composition $A_1\circ A_2\in I^{m_1+m_2}_{cl}(X\times Z, (C_1\circ C_2)';\Omega^{\frac{1}{2}})$. Then, 
\begin{align}\label{PScomp-eq1}
\sigma(B)\sim(a_1)_0(a_2)_0(\theta^2+\sigma^2)^{\frac{-n_Y}{2}} \sqrt{d\Phi}\in\Gamma^{m+n/4}(\Lambda,\mathscr{L}),
\end{align}
with $m=m_1+m_2$ and $n=n_X+n_Z.$ Here  $(a_1)_0$, $(a_2)_0$ are the principal parts of the amplitudes of $A_1$ and $A_2$, respectively.
\end{thm}
\begin{proof}
Since $e=0$, the intersection is transverse and $\Phi$ is a non-degenerate phase function. Then $\sqrt{d\Phi}$ is defined according to \thref{Dphi}. The definition of the principal symbol applied to $B=A_1\circ A_2$, tells us that
 \begin{align}\label{sigmaB}
     \sigma^{m}(B) \sim \widetilde{b_0}\sqrt{d\Phi},
 \end{align}
  with $b=(a_1)(a_2)(\theta^2+\sigma^2)^{\frac{-n_Y}{2}}$ the amplitude of $B$. The difficulty of the proof lies in verifying that $\sqrt{d\Phi}$ defines a section in $\Gamma(\mathscr{L},\Lambda)$, for $\Lambda=(C_1\circ C_2)'$. Luckily, all the necessary steps were proved in \cite{melin1975fourier}. First of all, notice that in general, the square of a section $\sigma\in\Gamma(\mathscr{L},\Lambda)$ defines, up to a sing, an almost analytic form on $\Lambda$ of maximal degree. That is \[\sigma^2\sim\pm\omega, \quad \text{for some}\ n\text{-form}\ \omega.\] 
By construction (see \thref{Dphi}) we know that, given a phase function $\phi$, there is a $n$-form $\omega=d\phi$, moreover 
\begin{align}\label{signo}
    \sigma^2\sim\pm d\phi, \quad \text{for}\ \sigma=\sqrt{d\phi}.
\end{align}
On the other hand, the proof of Theorem 7.5 in \cite{melin1975fourier} states that 
\[d\phi_1\land d\phi_2\sim \pm (\theta^2+\sigma^2)^{-n_Y} d\Phi\land \Omega\] defines an almost analytic form on $\Lambda=(C_1\circ C_2)'$, and that $\alpha_1^2\land\alpha_2^2\sim\pm \alpha^2\land\Omega$, for some $\Omega$. Here  $\alpha_1,\alpha_2,\alpha$ are the principal symbols of $A_1,\ A_2,$ and $B$, respectively. Combing these two facts with the expressions for $\alpha_1$, $\alpha_2$ according to \thref{PS2}, we get
 \begin{align*}
     \alpha_1^2\land\alpha_2^2&\sim \pm ( (a_1)^2_0\ d\phi_1)\land ( (a_2)^2_0\ d\phi_2) \sim \pm ((a_1)_0(a_2)_0)^2\ d\phi_1\land d\phi_2\\
     &\sim\pm\left((a_1)_0(a_2)_0(\theta^2+\sigma^2)^{-\frac{n_Y}{2}}\right)^2\ d\Phi\land \Omega,
 \end{align*}
Then, $\alpha^2\sim \pm b_0^2d\Phi$. From equation \eqref{signo}, we see that $\alpha\sim b_0 \sqrt{d\Phi}$. The result follow from equation \eqref{sigmaB} and the definition of $b_0$.
\end{proof}

Whenever \thref{CleanInt} is satisfied with $e>0$, the principal symbol of the resulting operator is given as in \thref{ps-clean-rem},  with $\sqrt{d\Phi}$ is defined according to \thref{Dphi-clean}. 
But, to compute this we need further assumptions. We take advantage of the identification between $C_1\circ C_2$ and $(C_1\times C_2)\cap \Tilde{D}$. Denote by $C$ the positive canonical relation $C_1\circ C_2\subseteq (T^*X\setminus 0)^\sim \times (T^*Z\setminus 0)^\sim.$
Then, the image $C_\gamma$ of a point $\gamma\in C_\R$ in $(C_{1\R}\times C_{2\R})\cap D$,  defines a fiber of dimension $e$ over $\gamma$.  We can now prove our main result, \thref{main}.

\begin{thm}
Let $A_1\in I^{m_1}_{\textup{cl}}(X\times Y, C_1;\Omega^{\frac{1}{2}})$ and $A_2\in I^{m_2}_{\textup{cl}}(Y\times Z, C_2;\Omega^{\frac{1}{2}})$ satisfy the assumptions of \thref{clean intersection}, and suppose that, for $\gamma\in C_\R$, the set $C_\gamma$ is compact. Then, the principal symbol of 
$B=A_1\circ A_2\in I^{m+e/2}_{\textup{cl}}(X\times Z, C;\Omega^{\frac{1}{2}})$, $m=m_1+m_2$, is 
\begin{align}\label{PScomp-eq2}
\sigma^{m+e/2}(B)\sim\int_{C_\gamma} \left((a_1)_0(a_2)_0(\theta^2+\sigma^2)^{\frac{-n_Y}{2}} \sqrt{d\Phi}\right)\ dy''d\theta''d\sigma''\in S^{(m-e/2+n/4)}(\Lambda,\mathscr{L}),
\end{align}
with $n=n_X+n_Z$ and $\sqrt{d\Phi}$ defined as in \thref{Dphi-clean}. 
\end{thm}

\begin{proof}
This is a direct consequence of \thref{PScomp} and \thref{ps-clean-rem}. We only need to compute the order of homogeneity. Since $\sqrt{d\Phi}$ is defined according to \thref{Dphi-clean}, it is homogeneous of degree $(N-e)/2$, with $N=n_Y+N_1+N_2$ and $e$ the excess of the intersection. Then, \[(a_1)_0(a_2)_0(\theta^2+\sigma^2)^{\frac{-n_Y}{2}} \sqrt{d\Phi}\in S^{(m'')}(\Lambda,\mathscr{L})\]
for $m''=m'+(N-e)/2=m_1+m_2-e/2+(n_X+n_Z)/4$. Equation \eqref{PScomp-eq2} follows after integration with respect to the excess variables $\omega''=\omega''(y,\theta,\sigma)$. Note that, it is possible to organize the variables in a way that $\omega''=(y'',\theta'',\sigma'')\in\R^e$, for some splitting $y=(y',y'')$, $\theta=(\theta',\theta'')$ and $\sigma=(\sigma',\sigma'')$.
\end{proof}

\begin{rem}
 \thref{CleanComp} implies that $\Lambda\sim (C_1\circ C_2)'$ is a positive Lagrangian manifold. As such, it can be parametrize by a non-degenerate phase function. This re-parametrization could potentially save us the difficulty of working with the excess variables. But, we need to keep in mind that the amplitude in the local representation will also change, which would  makes the formula \eqref{PScomp-eq2} no longer be valid. Instead, we would obtain a section of $\mathscr{L}$ which is equivalent to the representation of $\sigma^{m+e/2}(B)$ presented above.
\end{rem}

\textit{Acknowledgements.} The content of this paper is part of the author doctoral thesis. The work was funded by the DAAD  and supported by the activities of the RTG 2491: Fourier Analysis and Spectral Theory.  A special thanks goes to Ingo Witt for suggesting the topic and for the subsequent discussions and advice.

\printbibliography
\end{document}